\documentclass[a4paper,12pt]{article}
\usepackage{amsmath,amsthm,amssymb,threeparttable,epsfig}

\newtheorem{proposition}{\bf Proposition}[section]
\newtheorem{theorem}{\bf Theorem}[section]
\newtheorem{corollary}{\bf Corollary}[section]
\newtheorem{remark}{\bf Remark}[section]
\newcommand{\pair}[2]{\mbox{$\left\langle #1,#2 \right\rangle$}}
\def\argmin{\mathop{\mathrm{argmin}}}

\renewcommand{\labelenumi}

\title{On the convergence rate improvement of a splitting method for finding the resolvent of the sum of maximal monotone operators
\thanks{This work  was supported in part by the Ministry of Education, Culture, Sports, Science, and Technology [grant number 16K05280] }}
\author{Shin-ya Matsushita\thanks{Department of Intelligent Mechatronics, 
Akita Prefectural University, 84-4 Yuri-Honjo, Akita, Japan 
({\tt matsushita@akita-pu.ac.jp})} }

\begin{document}

\maketitle 

\begin{abstract}
This paper provides a new way of developing the splitting method  
which is used to solve the problem of finding 
the resolvent of the sum of maximal monotone operators 
in Hilbert spaces.  
By employing accelerated techniques developed by Davis 
and Yin (in Set-Valued Var. Anal. 25(4):829-858, 2017), 
this paper presents an implementable, 
strongly convergent splitting method which is designed to solve the problem. 
In particular, 
we show that the distance between the sequence of iterates and the solution converges to 
zero at a rate $O(1/k)$ to illustrate the efficiency 
of the proposed method, where $k$ is the number of iterations. 
Then, we apply the result to a class 
of optimization problems.
\end{abstract}




\pagestyle{plain}
\thispagestyle{plain}

\section{Introduction}
The purpose of this paper is to present a new splitting 
method for solving the following monotone inclusion 
problem:
\begin{equation}
\label{Sum}
\mbox{find}~u\in H~\mbox{such that}~z\in (I+A+B)(u), 
\end{equation}
where $A\colon H\rightrightarrows H$ and $B\colon H \rightrightarrows H$ are 
maximal monotone operators on a real Hilbert space $H$, 
$I$ is the identity mapping on $H$ and 
$z\in H$ is given.  
Problem (\ref{Sum}) has been widely studied in various 
fields such as sparse signal recovery  and best approximation problems; 
see \cite{Zou-Hastie,Combettes,Combettes2,C-P,Deutsch,B-C,Escalante-Raydan,H-L-Y,B-S,G-H-Y,G-Y-Z} and 
the references therein.
If the solution set of problem (\ref{Sum}) is nonempty
(which is assured if, for example $A+B$ is maximal monotone), 
the monotonicity of  $A+B$ guarantees 
uniqueness of the solution. 
We denote by $u=J_{A+B}(z)(:=(I+A+B)^{-1}(z))$
the solution of problem (\ref{Sum}) and 
$J_{A+B}$ is called the resolvent of $A+B$ (see, e.g.,
\cite[Subchapter 4.6]{T1}, \cite[Definition 23.1]{B-C}). 
Throughout this paper we assume that the resolvents of $A$ 
and $B$ are easy to compute. 

An interesting way of dealing with problem (\ref{Sum}) 
is to transform it into a problem with a particular structure. 
Arag\'on Artacho and Campoy  
have shown that problem (\ref{Sum}) can be transformed into a 
problem of finding a zero of the sum of two  maximal strongly monotone operators 
\cite[Proposition 3.2]{Aragon-Campoy2}. 
They then 
developed the averaged alternating modified reflections method 
for solving problem (\ref{Sum}) by applying  the 
Douglas-Rachford splitting method \cite{Lions-Mercier} 
to the latter problem \cite[Theorem 3.1]{Aragon-Campoy2}. 
The main advantage of their method is that 
it generates the sequence of iterates which converges 
strongly to the solution. 
However, 
it seems that 
the estimate of convergence rate for the method 
has not been considered  and 
the Douglas-Rachford splitting method can be slow 
\cite[Sections 6 and 7]{B-B-N-P-W} 
(see also \cite[Subsection 3.4]{Davis-Yin2} for related results).
Motivated by this fact,  
new techniques should be developed 
for analyzing
the convergence of iterative methods for solving (\ref{Sum}).

The goal of this paper is to propose an implementable, strongly convergent method 
for solving problem (\ref{Sum}) which has global convergence rates. 
In order to present a method, 
the structure in the 
transformed problem  of (\ref{Sum}) needs to be exploited. 
This can be done by using ideas from \cite{Aragon-Campoy2,Davis-Yin2}.  
The contributions of this paper can be summarized as follows. 
Firstly, we investigate the properties of an operator
$(A_{z})^{(\beta)}\colon H\rightrightarrows H$ defined by 
\begin{equation}
\label{NM}
(A_{z})^{(\beta)}:=2(1-\beta)A\left(\frac{1}{\beta}I+z\right)+\frac{1-\beta}{\beta}I, 
\end{equation}
for some $\beta\in (0,1)$. 
The operator $(A_{z})^{(\beta)}$ can be viewed as 
a modification of the inner $z$-perturbation and the $\beta$-strengthening of $A$ introduced in 
\cite[Definitions 2.3 and 3.2]{Aragon-Campoy2}\footnote{In \cite{Aragon-Campoy2}, 
the operator 
$A\left(\frac{1}{\beta}I-z\right)+\frac{(1-\beta)}{\beta}$
was treated as the 
inner $z$-perturbation and the $\beta$-strengthening $(A_{z})^{(\beta)}$ of $A$. 
We use (\ref{NM}) for notational convenience.}. 
It can be shown that  $(A_{z})^{(\beta)}$  is 
$\frac{1-\beta}{\beta}$-strongly monotone, and 
$J_{A+B}(z)=\frac{1}{\beta}v+z$ 
if and only if 
$0\in \left((A_{z})^{(\beta)}+(B_{z})^{(\beta)})\right)(v)$ 
\cite[Propositions 3.1 and 3.2]{Aragon-Campoy2}. 
Thus, problem (\ref{Sum}) can be transformed into 
the problem of finding a zero of the sum of two maximal strongly monotone 
operators $(A_{z})^{(\beta)}$   and $(B_{z})^{(\beta)}$. 
We consider 
the relation between the resolvents of $A$ and 
$r(A_{z})^{(\beta)}$, where $r>0$. 
In particular, we will show that the resolvent of $r(A_{z})^{(\beta)}$ 
can be obtained by evaluating the resolvent of $A$  in the 
original problem and thus the resolvent of the translated problem can be calculated. 

Secondly, 
we consider an accelerated variant of 
the three operator splitting method
developed in \cite[Algorithm 3]{Davis-Yin2}, 
which is designed to solve inclusion problems 
with three maximal monotone operators. 
Their method is conceptually very simple, but seems to be implementable 
only for the limited classes of problems where at least one operator is 
strongly monotone. 
We present a strongly convergent splitting method which is designed to 
solve problem (\ref{Sum}) by applying 
the method in \cite[Algorithm 3]{Davis-Yin2} to the translated problem. 
By exploiting
 special properties of the operators $(A_{z})^{(\beta)}$ and 
$(B_{z})^{(\beta)}$, 
the method can be applied without modifying the 
properties of $A$ and $B$ in the original problem.
Moreover, the proposed method involves the evaluation of the resolvents 
$r_k(A_{z})^{(\beta)}$ and $r_k(B_{z})^{(\beta)}$, and 
contains the parameter $\{r_k\}$ which has to 
vary  at each step to get better efficiency. 
This is in contrast with 
the averaged alternating modified reflections algorithm \cite{Aragon-Campoy2}, 
which uses similar resolvents with a constant parameter, 
for solving problem (\ref{Sum}). 
It follows from the fact mentioned above that 
the resolvent $r_k(A_{z})^{(\beta)}$ (resp. $r_k(B_{z})^{(\beta)}$) 
can be obtained by evaluating the resolvent of $A$  (resp. $B$). 
Thus  the proposed splitting method can be implemented and may be considered 
as a modification of the 
method in \cite{Davis-Yin2}. 
In particular, we can provide a $O(1/k)$ rate of convergence  and  
a strong convergence result for the sequence of iterates. 

Finally, we apply the results to a class of optimization problems. 
Our theoretical analysis is general and can handle 
convex minimization problems with three objective functions. 
Note that two of the objective functions are not 
necessary  differentiable. 
As important applications we consider the problem of  
minimizing the sum of a nonsmooth strongly convex function and a nonsmooth 
weakly convex function under the assumption that the strong convexity 
constant is larger than the weak convexity constant, and the 
best approximation problem 
since these problems possess a special structure. 
The convergence results  based on the 
Douglas-Rachford splitting method applied to these problems were 
obtained in \cite[Theorems 5.1 and 5.2]{G-H-Y} and \cite[Theorem 4.1]{Aragon-Campoy1}, 
respectively. 
However, it does not seem obvious how to estimate the distance 
between the sequences of iterates and the solutions. 
As a whole, the proposed method can be implemented and may be considered 
an improved version of the methods given in \cite{G-H-Y,Aragon-Campoy1}. 
Indeed, we can show that 
the distance between the sequence of iterates and the solution converges to 
zero at a rate $O(1/k)$, where $k$ is the number of iterations.

The rest of this paper is organized as follows. In section 2, we recall some 
definitions and known results for further analysis. 
Then, we investigate some properties of the mapping $r(A_{z})^{(\beta)}$ 
in section 3, where $r>0$ and $(A_{z})^{(\beta)}$ is defined by 
(\ref{NM}). 
A new splitting method is presented, the convergence of the method
is shown, and the rate of convergence is derived in section 4. 
Then, concrete examples of (\ref{Sum}) 
are given and we show how the proposed method can be applied 
in section 5. Finally, we draw some 
conclusions in section 6.

\section{Basic definitions and preliminaries}

The following notation will be used in this paper: 
$\mathbb{R}$ denotes the set of real numbers; 
$\mathbb{N}$ denotes the set of nonnegative integers; 
$H$ denotes a real Hilbert space; for any $x,y\in H$, $\pair{x}{y}$ denotes the inner product of $x$ and $y$; 
for any $z\in H$, $\Vert z\Vert$ denotes the norm of $z$, i.e., $\Vert z\Vert=\sqrt{\pair{z}{z}}$; 
for mappings $T\colon H\rightarrow H$  and $U\colon H\rightarrow H$, 
$T\circ U$ denotes the composition of $T$ and $U$; 
for any $C\subset H$ and mapping $U:C\rightarrow C$, $\mbox{Fix}(U)$ 
denotes the fixed point set of $U$, i.e., $\mbox{Fix}(U)=\{x\in C: U(x)=x\}$;
$d(x,C)=\inf\{ \Vert x-y\Vert :y\in C)\}$ 
denotes the distance from any $x$ to $C$; 
$\mbox{\rm int}C$ denotes the interior of set $C$;
$\mbox{\rm cone}(C)$ denotes the conical hull of $C$; 
$\mbox{\rm sri} C$ denotes the strong relative interior of $C$,
i.e., $\mbox{\rm sri} C := \{x \in C : \mbox{\rm cone}(C-x) ~\mbox{\rm is a closed  
linear subspace of}~H\}$;  
for any set-valued operator $A\colon H\rightrightarrows H$,  
$\mbox{\rm dom}(A)$ denotes the domain of $A$, i.e., $\mbox{\rm dom}(A) =
 \{x\in H : A(x)\neq\emptyset\}$, 
$\mbox{\rm ran}(A)$ denotes the range of $A$, i.e., 
$\mbox{\rm ran}(A) =\bigcup\{A(x) :x\in \mbox{\rm dom}(A)\}$, 
$G(A)$ denotes the graph of $A$, i.e., $G(A) = \{(x, x^{*}) : x^{*}\in A(x)\}$; 
The set of zero points of $A$ is denoted by $A^{-1}(0)$ i.e.,  
$A^{-1}(0)=\{z\in \mbox{\rm dom}(A) : 0\in A(z)\}$.

A mapping $U:C\rightarrow C$ is said to be 
\begin{itemize}
\item[{\rm (i)}] {\it nonexpansive} if 
\[
 \Vert U(x)-U(y)\Vert\le \Vert x-y\Vert ~~(x,y\in C);
\]
\item[{\rm (ii)}] {\it firmly nonexpansive} if 
\[
\Vert U(x)-U(y)\Vert^{2}\le \pair{x-y}{U(x)-U(y)} ~~(x,y\in C). 
\]
\end{itemize}
In particular, $U$ is firmly nonexpansive if and only if $2U-I$ is
nonexpansive \cite[Proposition 4.2]{B-C}. 

A set-valued operator  $A\colon H\rightrightarrows H$ 
is said to be 
\begin{itemize}
 \item[{\rm (i)}] {\it monotone} if 
\[
\pair{x-y}{x^{*}-y^{*}}\ge 0 
~~((x, x^{*}), (y, y^{*})\in G(A));  
\]
\item[{\rm (ii)}] {\it maximal monotone} if $A$ is monotone and 
$A = B$ whenever $B\colon H \rightrightarrows H$ is a monotone mapping such that $G(A)\subset G(B)$. 
\end{itemize}
The maximal monotonicity of $A$ implies that 
$\mbox{\rm ran}(I + rA) = H$ for all $r > 0$. Then, we
can define the {\it resolvent} $J_{rA}$ of $rA$ by 
\begin{equation}
\label{Resolvent}
 J_{rA}(x) = \{z \in H : x \in z + rA(z)\} = (I + rA)^{-1}(x) 
\end{equation}
for all $x\in H$.  
It is well-known that the resolvent is firmly nonexpansive and hence is 
Lipschitz continuous (see, e.g., \cite{B-C,T1}).  The following is a useful 
characterization of zeros of the 
sum of two maximal monotone operators. 
\begin{proposition} \cite[Proposition 25.1]{B-C}
\label{B-C1}
Let $A$ and $B$ be monotone operators on $H$, 
and let $r>0$. Then
\[
(A+B)^{-1}(0)=J_{rA}(\mbox{\rm Fix}((2J_{rB}-I)\circ (2J_{rA}-I))). 
\]
\end{proposition}

Let $f:E\rightarrow (-\infty,\infty]$ be a proper, lower semicontinuous convex 
function. The domain of function $f$ is  
$\mbox{\rm dom}f:=\{x\in H:f(x)<\infty\}$. 
The epigraph of $f$ is the set $\mbox{\rm epi} f$ defined by
$\mbox{\rm epi}f=\{(x,r)\in H\times \mathbb{R} : f(x) \le r \}$. 
$f$ is said to be {\it strongly convex} with constant $\beta > 0$ 
if for any $x, y \in H$ and for any $\lambda \in (0, 1)$, we 
have
\[
f((1-\lambda)x+\lambda y)\le (1-\lambda)f(x)+\lambda f(y)-\frac{\beta\lambda(1-\lambda)}{2}
\Vert x-y\Vert^{2}. 
\]
$f$ is said to be {\it weakly convex} if for some $\omega > 0$, the 
function $f+\frac{\omega}{2}{\Vert\cdot \Vert^{2}}$ is convex. 
The {\it conjugate function} of $f$ is the function 
$f^{*}\colon H\rightarrow\mathbb{R}\cup\{\infty\}$ defined by
$f^{*}(v)=\sup\{\pair{x}{v}-f(x):x\in \mbox{\rm dom}f\}$ 
 for $v\in H$. 
The {\it subdifferential} of $f$ at $x\in E$ is defined by
\[
 \partial f(x)=\{x^{*}\in H:f(y)\ge f(x)+\pair{y-x}{x^{*}}~\mbox{for all}~y\in H\}. 
\]
We know the subdifferential of a proper, lower semicontinuous and convex function
is maximal monotone (see, e.g., \cite{Rockafellar2}, \cite[Theorem 20.40]{B-C}). 
 Using the properties of subdifferentials, 
we can write (\ref{Resolvent})  equivalently as
\begin{equation}
\label{hoge5} 
J_{r\partial f}(x)
=\argmin_{y\in H}\left\{f(y)+\frac{1}{2r}\Vert y-x\Vert^{2}\right\}
\end{equation}
and (\ref{hoge5}) is known as the proximal mapping of 
$f$ \cite[Proposition 16.34]{B-C}.  
In particular, we denote by $\mbox{\rm prox}_{rf}(x)$ 
the proximal mapping of parameter $r$ at $x$ 
(i.e., $\mbox{\rm prox}_{rf}(x):=J_{r\partial f}(x)$). 

Let $C\subset H$ be a nonempty set. 
The {\it indicator function} of $C$, 
$i_{C} : H \rightarrow \mathbb{R}\cup \{\infty\}$, 
is the function which takes the value $0$ on $C$ and $+\infty$ otherwise. 
The {\it support function} $\sigma_C$ is defined by $\sigma_C(x)=
\sup_{c\in C}\pair{c}{x}$ for $x\in H$. 
The subdifferential of the 
indicator function is the {\it normal cone} of $C$, that is 
$N_C(x) = \{u \in H : \pair{u}{y-x}\le 
0 ~(\forall y \in C)\}$, if $x \in C$ and $N_{C}(x) = \emptyset$ for $x \notin C$. 
The proximal mapping is indeed an extension of the metric projection.
In fact, let $f(x) = i_C(x)$, it holds 
\begin{equation}
\label{Metric}
J_{r N_C} (x)=J_{N_C}(x)=J_{\partial i_C} (x)= P_C(x) 
\end{equation}
for any $r>0$, where $P_C : H\rightarrow C$ denotes the {\it metric projection} on $C$ 
(see \cite[Example 23.3 and Example 23.4]{B-C}).

We state the Stolz-C\'esaro theorem, which will be used.

\begin{theorem} (Stolz-Ces\'aro theorem)
 Let $\{a_k\}$ and $\{b_k\}$ be two sequences of real numbers. 
If $b_k$ is positive, strictly increasing and unbounded 
and the following limit exists:
$\lim_{k\rightarrow\infty}\frac{a_{k+1}-a_{k}}{b_{k+1}-b_{k}}=l$,
then the limit $\lim_{k\rightarrow\infty}\frac{a_k}{b_k}$
exists and it is equal to $l$. 
\end{theorem}

\section{Some properties of $(A_{z})^{(\beta)}$}

In this section, we investigate the properties of $(A_{z})^{(\beta)}$ 
defined by (\ref{NM}). 
Let $S$ be the set of solutions of problem (\ref{Sum}), i.e., 
$S=\{u\in H:z\in (I+A+B)(u)\}$. Under $S\neq\emptyset$, 
the monotonicity of $A+B$ guarantees the uniqueness of the solution of 
problem (\ref{Sum}) and hence  
$S=\{J_{A+B}(z)\}$.

We introduce some fundamental properties for
$(A_{z})^{(\beta)}$ defined by (\ref{NM}). 

\begin{proposition} \cite[Propositions 3.1 and 3.2]{Aragon-Campoy2}
\label{Prop1}
 Let $A$ and $B$ be operators on $H$ and 
let $\beta\in (0,1)$, let $z\in H$, 
and  let $(A_{z})^{(\beta)}$ (resp. $(B_{z})^{(\beta)}$) 
be the mapping defined by (\ref{NM}). 
Then 
\begin{enumerate}
 \item If $A$ is monotone, then $(A_{z})^{(\beta)}$ is 
$\frac{1-\beta}{\beta}$-strongly monotone; 
\item If $A$ is maximal monotone, then $(A_{z})^{(\beta)}$ is 
maximal monotone; 
\item $J_{A+B}(z)=\frac{1}{\beta}v+z$ 
if and only if 
$0\in \left((A_{z})^{(\beta)}+(B_{z})^{(\beta)})\right)(v)$.
\end{enumerate}
\end{proposition}

We consider the resolvent of $r(A_z)^{(\beta)}$ with $r>0$. 
Our method in the next section need to vary the parameter $r$ at each step. 
The following result is important to present an implementable method 
which is designed to find the solution to problem (\ref{Sum}).

\begin{proposition}
\label{Main1-1}
Let $A$ be a maximal monotone operator on $H$ and 
let $\beta\in (0,1)$, let $z\in H$, 
let $(A_{z})^{(\beta)}$ 
be the mapping defined by (\ref{NM}), and 
let $r>0$. 
Then for any $x\in H$, 
$J_{r(A_{z})^{(\beta)}}(x)=
\beta J_{\frac{2r(1-\beta)}{\beta +r(1-\beta)}A}
\left(\frac{1}{\beta +r(1-\beta)}x+z\right)-\beta z$. 
\end{proposition}
\begin{proof}
Let $u:=J_{r(A_{z})^{(\beta)}}(x)=(I+r(A_z)^{(\beta)})^{-1}(x)$.
This together with the definition  of ${(A_{z})^{(\beta)}}$  implies  that 
\begin{align*}
x&\in  u+2r(1-\beta)A\left(\frac{1}{\beta}u+z\right)
+\frac{1-\beta}{\beta}ru\\
&=\frac{\beta+r(1-\beta)}{\beta}u +2r(1-\beta)A\left(\frac{1}{\beta}u+z\right)\\
&=(\beta+r(1-\beta))\left( \frac{1}{\beta}u+z\right)+2r(1-\beta)A\left(\frac{1}{\beta}u+z\right)
-(\beta+r(1-\beta))z.
\end{align*}
Thus we have 
\begin{align*}
\frac{1}{\beta+r(1-\beta)}x+z
&\in \frac{1}{\beta}u+z+\frac{2r(1-\beta)}{\beta+r(1-\beta)}A
\left(\frac{1}{\beta}u+z\right)\\
&=\left( I+\frac{2r(1-\beta)}{\beta+r(1-\beta)}A\right)
\left(\frac{1}{\beta}u+z\right),
\end{align*}
and hence 
\begin{align*}
u&=\beta \left( I+
\frac{2r(1-\beta)}{\beta+r(1-\beta)}A\right)^{-1}\left(\frac{1}
{\beta+r(1-\beta)}x+z\right)-\beta z\\
&=\beta J_{\frac{2r(1-\beta)}{\beta +r(1-\beta)}A}
\left(\frac{1}{\beta +r(1-\beta)}x+z\right)-\beta z. 
\end{align*}
\end{proof}

\begin{remark}
 Arag\'on Artacho and Campoy \cite[Proposition 3.1]{Aragon-Campoy2} showed that 
the resolvent of  $A\left(\frac{1}{\beta}I\right)+\frac{1-\beta}{\beta}I$ is 
$\beta J_{A}$. Proposition \ref{Main1-1} enhances this result. 
\end{remark}

We next consider the existence of the solution of problem (\ref{Sum}). 
Using Propositions \ref{B-C1} and \ref{Prop1}, 
we establish a new connection between the existence of fixed points 
for  nonexpansive mappings and the solvability of problem (\ref{Sum}). 

\begin{theorem}
\label{Main1}
 Let $A$ and $B$ be 
maximal monotone operators on $H$ and 
let 
\begin{equation}
\label{NM2}
T:=(2J_{r(B_{z})^{(\beta)}}-I)\circ (2J_{r(A_{z})^{(\beta)}}-I),
\end{equation}
where $J_{r(A_{z})^{(\beta)}}$ 
(resp. $J_{r(B_{z})^{(\beta)}}$) is the resolvent 
of $r(A_{z})^{(\beta)}$  (resp. $r(B_{z})^{(\beta)}$) 
for some $\beta\in (0,1)$ and $r>0$. Then 
\begin{enumerate}
\item[{\rm (i)}] $\mbox{Fix}(T)\neq\emptyset$ if and only if $S\neq\emptyset$;
\item[{\rm (ii)}] $S=\beta \left(J_{r(A_{z})^{(\beta)}}(x)(\mbox{Fix}(T)) -z\right)$.
\end{enumerate}
\end{theorem}
\begin{proof}

(1) Let $u\in \mbox{\rm Fix}(T)$. 
It follows from Proposition \ref{B-C1} that 
\[
\left((A	_{z})^{(\beta)}+(B_{z})^{(\beta)}\right)^{-1}(0)=
J_{r(A_{z})^{(\beta)}}
(\mbox{\rm Fix}(T)). 
\]
Let $v:=J_{r(A_{z})^{(\beta)}}(u)\in \left((A_{z})^{(\beta)}+(B_{z})^{(\beta)}\right)^{-1}(0)$. 
From Proposition \ref{Prop1} (3), we have 
\[
J_{A+B}(z)=\frac{1}{\beta}v+z,
\]
and hence $\{\frac{1}{\beta}v+z\}=S$.

For the converse, let $u\in S$. Then define $v:=\beta (J_{A+B}(z)-z)$. 
It follows from Propositions \ref{B-C1} and \ref{Prop1} that 
\[
v\in \left((A	_{z})^{(\beta)}+(B_{z})^{(\beta)}\right)^{-1}(0)=
J_{r(A_{z})^{(\beta)}}\left(\mbox{\rm Fix}(T)\right).  
\]
Therefore, we conclude that $\mbox{Fix}(T)\neq\emptyset$. 

(2) From the arguments in the proof of (1), the result is obtained. 
\end{proof}

\begin{remark}
Theorem \ref{Main1} 
provides a new necessary and sufficient condition that 
guarantees the existence of $J_{A+B}(z)$. 
The advantage of our results is that 
the existence of solution of 
problem (\ref{Sum}) can be interpreted as a fixed point 
problem for nonexpansive mapping $(2J_{r(B_{z})^{(\beta)}}-I)\circ (2J_{r(A_{z})^{(\beta)}}-I)$.  
Hence, some existing results 
depending on the nonexpansiveness of a mapping are applicable. 
\end{remark}

By employing the classical result in \cite[Theorem 1]{B-P} (see also 
\cite[Theorem 3.1.6]{T1}), 
we prove the following result.

\begin{corollary}
\label{Cor1}
 Let $A$ and $B$ be 
maximal monotone operators on $H$ and 
let $T$ be defined by (\ref{NM2}). 
Then the following are equivalent:
\begin{itemize}
 \item[{\rm (i)}] There exists $x\in H$ such that $\{T^{k}(x)\}$ is bounded;
\item[{\rm (ii)}] $S\neq\emptyset$. 
\end{itemize}
\end{corollary}
\begin{proof}
Since 
$T=(2J_{r(B_{z})^{(\beta)}}-I)\circ (2J_{r(A_{z})^{(\beta)}}-I)$ is 
nonexpansive, by using the result in \cite[Theorem 1]{B-P}, 
$\{T^{k}(x)\}$ is bounded if and only if $\mbox{\rm Fix}(T)\neq\emptyset$. 
It follows from Theorem \ref{Main1} that 
$\{T^{k}(x)\}$ is bounded if and only if $S\neq\emptyset$. 
\end{proof}

\begin{remark}
The maximal monotonicity of $A+B$ guarantees the 
existence of the solution of problem (\ref{Sum}). 
Various qualification conditions 
have been presented in the literature
to prove maximality of the sum of two maximal 
monotone operators (see \cite[Theorems 1 and 2]{Rockafellar1}, 
\cite[Subchapter 24.1]{B-C}). 
For example, if 
\begin{equation}
\label{CQ}
 0\in \mbox{\rm sri}(\mbox{\rm dom}A-\mbox{\rm dom}B)
\end{equation}
holds, then $A+B$ is maximal monotone. 
Moreover, (\ref{CQ}) holds if one of the following condition holds 
\begin{itemize}
 \item[{\rm (1)}] $\mbox{\rm dom}B=H$; 
\item[{\rm (2)}] $\mbox{dom}A\cap \mbox{\rm int~dom}B\neq\emptyset$; 
\item[{\rm (3)}] $0\in \mbox{int}(\mbox{\rm dom}A-\mbox{\rm dom}B)$; 
\end{itemize}
(see, e.g., \cite[Corollary 24.4]{B-C}). Thus,  
these conditions guarantee the 
existence of the 
solution of (\ref{Sum}). However, the solution set $S$ may be empty when 
(\ref{CQ}) does not hold and the difficulty is how to 
check that such condition holds. 
Corollary \ref{Cor1} shows that $\{T^{k}(x)\}$ 
can be used to determine the existence of the solution of problem (\ref{Sum}). 
\end{remark}

\section{Convergence analysis}

In this section, we will propose a splitting method to solve problem (\ref{Sum}). 
Let $A$ and $B$ be maximal monotone operators on $H$. 
Assume that $\beta\in (0,1)$,  $z_{0}\in H$, 
$x_{0}=\beta J_{\frac{2r_0(1-\beta)}{\beta+r_0(1-\beta)}A}\left(
\frac{1}{\beta+r_0(1-\beta)}z_{0}+z\right)-\beta z_{0}$, 
$y_{0}=(1/r_0)(z_0-x_0)$
and let $\{x_k\}$, $\{y_k\}$  and $\{z_k\}$ be the sequences generated
by
\begin{equation}
\label{New1}
\left \{
\begin{array}{l}
x_{k}=\beta J_{\frac{2r_{k-1}(1-\beta)}{\beta +r_{k-1}(1-\beta)}A}
\left(\frac{1}{\beta +r_{k-1}(1-\beta)}
(z_{k-1}+r_{k-1}y_{k-1})+z\right)-\beta z,\\
y_{k}=(1/r_{k-1})(z_{k-1}+r_{k-1}y_{k-1}-x_{k}),\\
z_{k}=
\beta J_{\frac{2r_{k}(1-\beta)}{\beta +r_{k}(1-\beta)}B}\left(\frac{1}{\beta +r_{k}(1-\beta)}
(x_{k}-r_{k}y_{k})+z\right)-\beta z.
\end{array}
\right.
\end{equation}
where $\{r_k\}$  is a sequence of positive real numbers 
such that  
\begin{equation}
\label{C1}
r_0\in (0,2(1-\beta)/\beta )~\mbox{and}~
r_{k+1}=r_{k}/\sqrt{1+2r_k(1-\beta)/\beta}.
\end{equation}

We can provide convergence results and rates
for the sequence $\{x_k\}$ in (\ref{New1}).

\subsection{Connections to other existing methods}

In this subsection, we present the 
connections of the proposed iterative method (\ref{New1}) to 
existing iterative methods.

By Proposition \ref{Main1}, 
(\ref{New1}) can be stated equivalently as $x_0=J_{r_0 (A_z)^{(\beta)}}(z_0)$, 
$y_0=(1/r_0)(I-J_{r_0 (A_z)^{(\beta)}})(z_0)$ and 
\begin{equation}
\label{New2}
\left \{
\begin{array}{l}
x_{k}=J_{r_{k-1}(A_z)^{(\beta)}}(z_{k-1}+r_{k-1}y_{k-1}),\\
y_{k}=(1/r_{k-1})(z_{k-1}+r_{k-1}y_{k-1}-x_{k}),\\
z_{k}=J_{r_{k}(B_z)^{(\beta)}}(x_{k}-r_{k}y_{k}),~k=1,2,\dots. 
\end{array}
\right.
\end{equation}
(\ref{New2}) can be considered as an instance of the iterative method for 
solving the problem of finding a zero of the sum of monotone operators 
developed by Davis and Yin \cite[Algorithm 3]{Davis-Yin2}.  
More precisely, we apply their method to 
the problem of finding a point $v\in H$ such that 
\begin{equation}
\label{Sum2}
0\in ((A_z)^{(\beta)}+(B_z)^{(\beta)})(v), 
\end{equation}
where $(A_z)^{(\beta)}$ and $(B_z)^{(\beta)}$ 
are defined by (\ref{NM}). 
The main difficulties in implementing (\ref{New2}) 
lies in the fact that it involves 
the evaluation of  the resolvents $J_{r_k(A_z)^{(\beta)}}$  and $J_{r_k(B_z)^{(\beta)}}$, and  
contains the parameter $\{r_k\}$ which has to be adjusted adaptively
at each iteration. 
Using Proposition \ref{Main1}, (\ref{New2})
can be implemented by using the resolvents  of 
$A$ and $B$. 
In particular, we will show that the sequence $\{({1}/{\beta})x_k+z\}$ 
converges strongly to 
$J_{A+B}(z)$, and
$
\Vert ({1}/{\beta})x_{k+1}+z-J_{A+B}(z)\Vert=O(1/k)
$
holds under condition (\ref{C1}). 
Thus (\ref{New1}) can 
considered as the modification of the method in \cite{Davis-Yin2}.

Next, we consider the connection between (\ref{New1}) and the 
Douglas-Rachford splitting method \cite{Lions-Mercier}. 
The Douglas-Rachford splitting method has the following form:
\begin{equation}
\label{DR}
w_{k+1}=w_{k}+\lambda_{k}(J_{\gamma B}\circ (2J_{\gamma A}-I)(w_{k})-J_{\gamma A}(w_{k}))
\end{equation}
where $w_{0}\in H$, $\gamma \in (0,\infty)$ and 
$\{\lambda_k\}\subset [0,2]$. 
The iterative scheme (\ref{DR}) can be applied to 
solve the inclusion $0\in (A+B)(u)$. 
A general discussion on the Douglas-Rachford method 
can be found in \cite[Subchapter 25.2]{B-C}. 
In (\ref{New2}), we use a fixed parameter $r_k:=r>0$. Now we define 
$u_{k+1}:= z_{k} + ry_{k}$.  Then we have 
\begin{align*}
u_{k+1}&=z_{k} + ry_{k}\\
&=J_{r(B_z)^{(\beta)}}(x_{k}-ry_{k})+
z_{k-1}+ry_{k-1}-x_{k}\\
&=u_{k}+
J_{r(B_z)^{(\beta)}}(2J_{r(A_z)^{(\beta)}}( u_{k})-u_{k})
-J_{r(A_z)^{(\beta)}}( u_{k})\\
&=u_{k}+J_{r (B_z)^{(\beta)}}\circ (2J_{r (A_{z})^{(\beta)}}-I)(u_{k})-
J_{r (A_z)^{(\beta)}}(u_{k}). 
\end{align*}
Thus, the sequence $\{u_k\}$ can be viewed 
as a special case of (\ref{DR}) for solving (\ref{Sum2}) 
when we keep the parameter $r_k$ fixed. 

On the other hand, (\ref{DR}) is 
closely related to 
the averaged alternating modified reflections algorithm
in \cite{Aragon-Campoy2}. 
Arag\'on Artacho and Campoy  
considered the following iterative scheme:
\begin{equation}
\label{AAMRA} 
v_{k+1} = (1-\lambda_{k})v_{k} + \lambda_{k}\left(2J_{\left(\frac{\gamma}{2(1-\beta)} B_z\right)^{(\beta)}}-I\right)\circ
\left(2J_{\left(\frac{\gamma}{2(1-\beta)} A_z\right)^{(\beta)}}-I\right)(v_k),
\end{equation}
where $v_0\in H$, $\gamma>0$ and $\{\lambda_k\}\subset [0,1]$ such that 
$\sum_{j=0}^{\infty}\lambda_{j}(1-\lambda_j)=\infty$. Note that 
(\ref{AAMRA}) is equivalently written as 
\begin{equation}
\label{AAMRA2} 
v_{k+1} = v_{k} + 2\lambda_{k}\left(J_{\left(\frac{\gamma}{2(1-\beta)} B_z\right)^{(\beta)}}\circ
J_{\left(\frac{\gamma}{2(1-\beta)} A_z\right)^{(\beta)}}(v_k)-J_{\left(\frac{\gamma}{2(1-\beta)} A_z\right)^{(\beta)}}(v_k)\right)
\end{equation}
(see \cite[Proposition 4.21]{B-C}), and hence the 
averaged alternating modified reflections algorithm can be viewed as 
a special case of (\ref{DR}) applied to solve 
\[
0\in \left(\left(\frac{\gamma}{2(1-\beta)} A_z\right)^{(\beta)}+\left(\frac{\gamma}{2(1-\beta)} B_z\right)^{(\beta)}\right)(v). 
\]
It is shown in \cite[Theorem 3.1]{Aragon-Campoy2} that 
$\{J_{\gamma A}(v_k+z)\}$ converges strongly to $J_{\frac{\gamma}{2(1-\beta)}(A+B)}(z)$
when $z\in \mbox{\rm ran}\left(I+\left(\gamma/2(1-\beta)\right)(A+B)\right)$.
Instead of fixing the parameter, the varying sequence of parameters $\{r_k\}$ 
is used in our proposed method (\ref{New1}). Thus, (\ref{New1}) is different but closely related to 
the averaged alternating modified reflections algorithm.

\subsection{Convergence of (\ref{New1})}
The following theorem concerns the strong convergence 
and convergence rate of the sequence $\{(1/\beta)x_{k+1}+z\}$,  
where $\{x_k\}$ is generated by (\ref{New1}). 
We first prove a proposition which plays important 
roles in the convergence analysis.

\begin{proposition} 
Let $A$ and $B$ be maximal monotone operators 
such that $S\neq\emptyset$, 
and let $\{x_k\}$, 
$\{y_{k}\}$ and $\{z_{k}\}$ be the sequences
generated by (\ref{New1}) (or equivalently (\ref{New2})). 
Then the following inequality holds:
\begin{equation}
\label{G1}
(1/r_{k+1}^{2}) \Vert x_{k+1}-v\Vert^{2}
+\Vert y_{k+1}-v_A\Vert^{2}
\le (1/r_{k}^{2})\Vert x_{k}-v\Vert^{2}+\Vert y_k-v_A\Vert^{2},
\end{equation}
where $r>0$, $u\in \mbox{\rm Fix}(T)$, $v=J_{r\left(A_{z}\right)^{(\beta)}}(u)$, $v_{A}=(1/r)(u-v)$ and 
$v_{B}=(1/r)(v-u)$ such that 
$v_{A}\in (A_{z})^{(\beta)}(v)$ and $v_{B}\in (B_{z})^{(\beta)}(v)$. 
\end{proposition}
\begin{proof}
The proof is similar to \cite[Proposition 3.1]{Davis-Yin2}, 
however, for the convenience of the reader, we sketch it here. 
From the definition of $\{z_k\}$ in (\ref{New2}), we have 
\[
 x_k-r_k y_k\in z_k+r_k \left(B_{z}\right)^{(\beta)}(z_k). 
\]
Let 
\[
 v_k:=(1/r_k)(x_k-r_k y_k-z_k)\in \left(B_{z}\right)^{(\beta)}(z_k).
\]
It follows from the definitions of $\{v_k\}$ and $\{y_k\}$, we have 
\begin{equation}
\label{D1}
 r_k (y_{k+1}-y_k)=z_k-x_{k+1},  
\end{equation}
\begin{equation}
\label{D2}
r_{k}(y_{k+1}+v_k)=z_k+r_ky_k-x_{k+1}+x_k-r_ky_k-z_k=x_k-x_{k+1},
\end{equation}
and 
\begin{equation}
\label{D3}
r_{k}(y_{k}+v_k)=r_ky_k+x_k-r_ky_k-z_k=x_k-z_{k}. 
\end{equation}
By using (\ref{D1}), (\ref{D2}) and (\ref{D3}), we have 
\begin{align}
&2r_{k}\left(\pair{z_{k}-v}{v_k}+\pair{x_{k+1}-v}{y_{k+1}}\right)\notag\\
=&2r_{k}\left(\pair{z_{k}-x_{k+1}}{v_k}+\pair{x_{k+1}-v}{y_{k+1}+v_k}\right)\notag\\
=&2r_{k}\left(\pair{z_{k}-x_{k+1}}{v_k+y_{k}}-\pair{z_{k}-x_{k+1}}{y_{k}}
\right)+2\pair{x_{k+1}-v}{x_{k}-x_{k+1}}\notag\\
=&2\pair{z_{k}-x_{k+1}}{x_{k}-z_{k}}+2\pair{x_{k+1}-v}{x_{k}-x_{k+1}}\notag\\
&+2r_{k}\pair{z_{k}-x_{k+1}}{v_{A}-y_{k}}-2r_k\pair{z_{k}-x_{k+1}}{v_{A}}\notag\\
=&2\pair{z_{k}-x_{k+1}}{x_{k}-z_{k}}+2\pair{x_{k+1}-v}{x_{k}-x_{k+1}}\label{F1} \\
&+2r_{k}^{2}\pair{y_{k+1}-y_{k}}{v_{A}-y_{k}}-2r_k\pair{z_{k}-x_{k+1}}{v_{A}}.\notag
\end{align}
Applying the relation
\[
2\pair{a-b}{c-a} = \Vert b -c\Vert^{2}-\Vert a -c\Vert^{2}-\Vert b- a\Vert^{2}
\]
to (\ref{F1}), we have 
\begin{align}
&2r_{k}\left(\pair{z_{k}-v}{v_k}+\pair{x_{k+1}-v}{y_{k+1}}\right)\notag\\
=&\Vert x_{k+1}-x_{k}\Vert^{2}- \Vert z_{k}-x_{k}\Vert^{2}-\Vert x_{k+1}-z_k\Vert^{2}\notag\\
&+\Vert v-x_k\Vert^{2}-\Vert x_{k+1}-x_{k}\Vert^{2}- \Vert v-x_{k+1}\Vert^{2}\notag\\
&+r_k^{2}(\Vert y_k-v_A\Vert^{2}-\Vert y_k-v_{A}\Vert^{2} -\Vert y_k-y_{k+1}\Vert^{2})
-2r_k\pair{z_{k}-x_{k+1}}{v_{A}}\notag\\
=&\Vert x_{k}-v\Vert^{2}-\Vert x_{k+1}-v\Vert^{2}-\Vert z_k-x_{k}\Vert^{2}\label{E1} \\
&+r_{k}^{2}\Vert y_{k}-v_{A}\Vert^{2}-r_k^{2}\Vert y_{k+1}-v_{A}\Vert^{2}
+2r_{k}\pair{x_{k+1}-z_{k}}{v_{A}}.\notag
\end{align}
On the other hand, 
by (\ref{F1}) and 
strong monotonicity of $(A_{z})^{(\beta)}$ and $(B_z)^{(\beta)}$, we have
\begin{align}
&2r_{k}\left(\pair{z_{k}-v}{v_k}+\pair{x_{k+1}-v}{y_{k+1}}\right)\notag\\
\ge &2r_{k}( \pair{z_k-v}{v_B}+(1-\beta)/\beta\Vert 
z_{k}-v\Vert^{2}\notag\\
&+\pair{x_{k+1}-v}{v_A}+(1-\beta)/\beta\Vert 
x_{k+1}-v\Vert^{2}
)\notag\\
=&
2r_{k}( \pair{x_{k+1}-z_k}{v_A}+(1-\beta)/\beta\Vert 
z_{k}-v\Vert^{2}+(1-\beta)/\beta\Vert 
x_{k+1}-v\Vert^{2}
). \label{E2}
\end{align}
By using (\ref{E1}) and (\ref{E2}) we obtain
\begin{align*}
 &2r_{k}( \pair{x_{k+1}-z_k}{v_A}+(1-\beta)/\beta\Vert 
z_{k}-v\Vert^{2}+(1-\beta)/\beta\Vert 
x_{k+1}-v\Vert^{2}
)\\
\le &
\Vert x_{k}-v\Vert^{2}-\Vert x_{k+1}-v\Vert^{2}-\Vert z_k-x_{k}\Vert^{2}\\
&+r_{k}^{2}\Vert y_{k}-v_{A}\Vert^{2}-r_k^{2}\Vert y_{k+1}-v_{A}\Vert^{2}
+2r_{k}\pair{x_{k+1}-z_{k}}{v_{A}},
\end{align*}
and hence
\begin{align*}
&(1+2r_{k}(1-\beta)/\beta) \Vert x_{k+1}-v\Vert^{2}
+r_{k}^{2}\Vert y_{k+1}-v_A\Vert^{2}\\
\le& \Vert x_{k}-v\Vert^{2}+r_k^{2}\Vert y_k-v_A\Vert^{2}. 
\end{align*}
Multiplying the inequality by $r_{k}^{2}$ and using (\ref{C1}), 
we get (\ref{G1}). 
\end{proof}

We prove the strong convergence of the sequence $\{x_k\}$
generated by  (\ref{New1}). 

\begin{theorem}
 \label{Main2}
Let $A$ and $B$ be maximal monotone operators and let 
$\{x_{k}\}$, $\{y_{k}\}$ and 
$\{z_{k}\}$ be the sequences
generated by (\ref{New1}). 
If $S\neq\emptyset$, then $\{(1/\beta))x_{k}+z\}$ converges strongly to 
$J_{A+B}(z)$. In particular, 
the following holds: 
\begin{equation}
\label{Estimate1} 
\Vert (1/\beta)x_{k+1}+z-J_{A+B}(z)\Vert=O(1/k). 
\end{equation}
\end{theorem}
\begin{proof}
We know  that 
$0<r_{k+1}<r_{k}<r_{0}<2(1-\beta)/\beta~(\forall k\in \mathbb{N})$. 
It follows that the sequence $\{r_k\}$ has the limit. 
Moreover, it follows from (\ref{C1}) that $\lim_{k\rightarrow\infty}r_k=0$. 
Hence, we can further get
\[
 \lim_{k\rightarrow\infty}\frac{r_{k}}{r_{k+1}}= \lim_{k\rightarrow\infty}\sqrt{1+2r_k(1-\beta)/\beta}=1. 
\]
This implies that  
\begin{align*}
\frac{(k+2)-(k+1)}{\frac{1}{r_{k+1}}-\frac{1}{r_{k}}}&=
\frac{r_{k}r_{k+1}}{r_k-r_{k+1}}=\frac{r_{k}r_{k+1}(r_{k}+r_{k+1})}{r_k^{2}-r_{k+1}^{2}}\\
&=\frac{r_{k}r_{k+1}(r_{k}+r_{k+1})}{2r_kr_{k+1}^{2}(1-\beta)/\beta}\\
&=\frac{r_k+r_{k+1}}{2r_{k+1}(1-\beta)/\beta}\\
&=\frac{r_{k}/r_{k+1}+1}{2(1-\beta)/\beta},
\end{align*}
and thus
\[
 \lim_{k\rightarrow\infty}\frac{(k+2)-(k+1)}{\frac{1}{r_{k+1}}-\frac{1}{r_{k}}}=
\frac{\beta}{1-\beta}.
\]
So, we can
use the Stolz-Ces\'aro theorem with 
$a_{k}:=k+1$ and $b_k:=\frac{1}{r_k}$ 
to conclude that $\{(a_k-a_{k-1})/(b_k-b_{k-1})\}$ and $\{a_k/b_k\}$ have
the same limit.

On the other hand, let $r>0$, $u\in \mbox{\rm Fix}(T)$ and 
$v:=J_{r(A_n)^{(\beta)}}(u)$. 
It follows from (\ref{G1}) that 
the following inequality holds for all $k\in \mathbb{N}\cup \{0\}$: 
\begin{align*}
&(1/r_{k+1}^{2})\Vert x_{k+1}-v\Vert^{2}
+\Vert y_{k+1}-v_A\Vert^{2}
\le (1/r_{k}^{2})\Vert x_{k}-v\Vert^{2}+\Vert y_k-v_A\Vert^{2}. 
\end{align*}
Thus, we have
\[
 \Vert x_{k+1}-v\Vert^{2}\le r_{k+1}^{2}((1/r_{0}^{2})\Vert x_0-v\Vert^{2}+\Vert
y_{0}-v_A\Vert^{2})=O(1/k^{2}). 
\]
From Theorem \ref{Main1} (2), we have 
\[
 \Vert (1/\beta)x_{k+1}+z-J_{A+B}(z)\Vert=
 \Vert (1/\beta)x_{k+1}+z-\left((1/\beta)v+z\right)\Vert
=O(1/k). 
\]
The proof is complete.
\end{proof}

\section{Applications}
In this section, we provide
some concrete problems that reduce to problem (\ref{Sum}).
We apply the proposed method (\ref{New1}) to a class of 
optimization problems consisting
of the sum of three functions. 
Let $z\in H$ and let 
$f,g\colon H\rightarrow (-\infty,\infty]$ be proper, lower semicontinuous and convex 
functions. We consider the following problem: 
\begin{equation}
\label{hoge3}
\mbox{minimize}~\frac{1}{2}\Vert x-z\Vert^{2}+ 
f(x)+g(x).
\end{equation}
We refer the reader to \cite{C-P,C-D-V} for more details and 
applications of problem (\ref{hoge3}) and its
useful variants in image processing.
The solution set of problem (\ref{hoge3}) coincides with the 
solution set of the monotone inclusion problem
\begin{equation*}
\mbox{find}~u\in H~\mbox{such that}~z\in (I+\partial (f+g))(u).
\end{equation*}
Under the condition that $\mbox{\rm dom}f\cap\mbox{\rm dom}g\neq
\emptyset$, the maximal monotonicity of $\partial (f+g)$ 
guarantees the existence and uniqueness of the solution of 
problem (\ref{hoge3}), denoted by $\mbox{\rm prox}_{f+g}(z)$ 
\cite[Theorem 4.6.5]{T1}, \cite[Proposition 16.35]{B-C}. 
It is important to point out that it holds 
$\mbox{\rm prox}_{f+g}(z)=J_{\partial f+\partial g}(z)$ 
when $J_{\partial f+\partial g}(z)$ exists 
\cite[Remark 3.4]{B-C2}. 

From the discussion in Sections 3 and 4 we get the following result. 
The proof is similar to that of 
Corollary \ref{Cor1} and Theorem \ref{Main1}, 
and thus is omitted.
\begin{corollary}
\label{Main3} 
Let $z\in H$ and let 
$f,g\colon H\rightarrow (-\infty,\infty]$ be proper, lower semicontinuous and convex 
functions with $\mbox{\rm dom}f\cap\mbox{\rm dom}g\neq\emptyset$. 
Assume that $\beta\in (0,1)$ and $\{r_k\}\subset (0,2(1-\beta)/\beta)$ 
such that (\ref{C1}) holds. 
Let $\{x_k\}$, $\{y_k\}$  and $\{z_k\}$ be the sequences generated
by 
$z_{0}\in H$, 
$x_{0}=\beta \mbox{\rm prox}_{\frac{2r_0(1-\beta)}
{\beta+r_0(1-\beta)}f}\left(
\frac{1}{\beta+r_0(1-\beta)}z_{0}+z\right)-\beta z_{0}$, 
$y_{0}=(1/r_0)(z_0-x_0)$
and 
\begin{equation}
\label{New3}
\left \{
\begin{array}{l}
x_{k}=\beta \mbox{\rm prox}_{\frac{2r_{k-1}(1-\beta)}{\beta +r_{k-1}
(1-\beta)}f}\left(\frac{1}{\beta +r_{k-1}(1-\beta)}
(z_{k-1}+r_{k-1}y_{k-1})+z\right)-\beta z,\\
y_{k}=(1/r_{k-1})(z_{k-1}+r_{k-1}y_{k-1}-x_{k}),\\
z_{k}=
\beta \mbox{\rm prox}_{\frac{2r_{k}(1-\beta)}{\beta +r_{k}(1-\beta)}g}
\left(\frac{1}{\beta +r_{k}(1-\beta)}
(x_{k}-r_{k}y_{k})+z\right)-\beta z.
\end{array}
\right.
\end{equation}
The following assertions hold: 
\begin{enumerate}
 \item [{\rm (i)}]$J_{\partial f+\partial g}(z)$ exists 
if and only if there exists $x\in H$ such that 
$\{T^{k}(x)\}$ is bounded, where $T$ is defined by (\ref{NM2}) with 
$A:=\partial f$ and $B:=\partial g$;
 \item[{\rm (ii)}] 
If $J_{\partial f+\partial g}(z)$ exists, then 
$\{(1/\beta)x_k+z\}$ converges strongly to $\mbox{\rm prox}_{ f+ g}(z)$, 
and the convergence rate estimate 
$\Vert (1/\beta)x_{k+1}+z-\mbox{\rm prox}_{f+g}(z)\Vert=O(1/k)$ holds. 
\end{enumerate}

\end{corollary}

\begin{remark}
\label{Re1}
Burachik and Jeyakumar \cite{B-J} showed that
the subdifferential sum formula 
$\partial (f+g)(x)=\partial f(x)+\partial g(x)~(\forall 
x\in \mbox{\rm dom}f\cap\mbox{\rm dom}g)$ holds 
whenever $\mbox{\rm epi} f^{*}+\mbox{\rm epi}g^{*}$ is weakly closed 
\cite[Theorem 3.1]{B-J}.   
Furthermore,  it was shown that 
$0\in \mbox{sri}(\mbox{\rm dom}f-\mbox{\rm dom}g)$ 
implies $\mbox{\rm epi} f^{*}+\mbox{\rm epi} g^{*}$ is weakly closed 
\cite[Proposition 3.2]{B-J}. 
Note that under the subdifferential sum formula, 
the assumption of the existence of $J_{\partial f+\partial g}(z)$ 
in Corollary \ref{Main3} can be removed. 
\end{remark}

\subsection{Minimizing the sum of 
 a strongly convex function and a weakly convex function}

We apply (\ref{New3}) to the minimization of two 
functions, where one is strongly convex and the other is weakly convex. 
Consider the following minimization problem:
\begin{equation}
\label{hoge7}
\mbox{minimize}~\tilde{f}(x)+\tilde{g}(x),
\end{equation}
where $\tilde{f}\colon H\rightarrow (-\infty,\infty]$ 
is proper lower semicontinuous strongly convex with constant $\gamma > 0$, 
and $\tilde{g} : H \rightarrow (-\infty,\infty]$ is proper lower 
semicontinuous weakly convex with constant $\omega > 0$. 
(\ref{hoge7}) contains signal and image processing problems; 
see, e.g., \cite{M-S-M-C,B1,B-S,G-H-Y,G-Y-Z}.

The convergence of the Douglas-Rachford splitting method for (\ref{hoge7}) 
was established in \cite{G-H-Y} under the assumption 
$\gamma > \omega$.  
In this case, problem (\ref{hoge7}) has the
unique solution and we denote the unique minimizer by $x^*$. 
The convergence results of the Douglas-Rachford splitting 
method for (\ref{hoge7}) were shown \cite[Theorems 4.4 and 4.6]{G-H-Y}, 
and the rates of asymptotic regularity for the corresponding Douglas-Rachford 
operators were derived \cite[Theorems 5.1 and 5.2]{G-H-Y} 
under appropriate assumptions.

It is assumed that $\gamma > \omega$  in our discussion. 
(\ref{hoge7}) is equivalent to the following problem:
\begin{equation}
\label{hoge8}
\mbox{minimize}~\frac{\gamma-\omega}{2}\Vert x\Vert^{2}+
\tilde{f}(x)-\frac{\gamma}{2}\Vert x\Vert^{2}+
\tilde{g}(x)+\frac{\omega}{2}\Vert x\Vert^{2}. 
\end{equation}
It follows from \cite[Exercise 12.59]{R-W} and 
the definition of the weakly convex function, 
$\tilde{f}-({\gamma}/{2})\Vert \cdot\Vert^{2}$ and 
$\tilde{g}+({\omega}/{2})\Vert \cdot\Vert^{2}$ are convex 
so that the  method (\ref{New3})
can be applied to problem (\ref{hoge8}). 
By letting $f:=(1/(\gamma-\omega))(\tilde{f}-(\gamma/2)\Vert\cdot \Vert^{2})$, 
$g:=(1/(\gamma-\omega))(\tilde{g}+(\omega/2)\Vert\cdot \Vert^{2})$ and $z:=0$, 
it holds that  $J_{\partial {f}+\partial{g}}(0)=x^*$ 
when $J_{\partial {f}+\partial{g}}(0)$ exists 
and 
hence $J_{\partial {f}+\partial{g}}(0)$ is a solution of (\ref{hoge7}). 
Now, we get the following result. 

\begin{corollary}
\label{Main4} 
Let $\tilde{f}\colon H\rightarrow (-\infty,\infty]$ 
be proper lower semicontinuous strongly convex with constant $\gamma > 0$, 
and $\tilde{g} : H \rightarrow \mathbb{R}\cup \{\infty\}$ be proper lower 
semicontinuous weakly convex with constant $\omega > 0$. 
Assume that $\gamma>\omega$, $\beta\in (0,1)$ and 
$\{r_k\}\subset (0,2(1-\beta)/\beta)$ 
such that (\ref{C1}) holds. 
Let $\{x_k\}$, $\{y_k\}$  and $\{z_k\}$ be the sequences generated 
by 
by 
$z_{0}\in H$, 
$x_{0}=\beta 
\mbox{\rm prox}_{\frac{2r_0(1-\beta)}{\beta+r_0(1-\beta)}f}\left(
\frac{1}{\beta+r_0(1-\beta)}z_{0}\right)+\beta z_{0}$, 
$y_{0}=(1/r_0)(z_0-x_0)$
and 
\begin{equation}
\label{New4}
\left \{
\begin{array}{l}
x_{k}=\beta \mbox{\rm prox}_{\frac{2r_{k-1}(1-\beta)}{\beta +r_{k-1}(1-\beta)}f}\left(\frac{1}{\beta +r_{k-1}(1-\beta)}
(z_{k-1}+r_{k-1}y_{k-1})\right),\\
y_{k}=(1/r_{k-1})(z_{k-1}+r_{k-1}y_{k-1}-x_{k}),\\
z_{k}=
\beta \mbox{\rm prox}_{\frac{2r_{k}(1-\beta)}{\beta +r_{k}(1-\beta)}g}\left(\frac{1}{\beta +r_{k}(1-\beta)}
(x_{k}-r_{k}y_{k})\right),
\end{array}
\right.
\end{equation}
where $f:=(1/(\gamma-\omega))(\tilde{f}-(\beta/2)\Vert\cdot \Vert^{2})$ 
and $g:=(1/(\gamma-\omega))(\tilde{g}+(\omega/2)\Vert\cdot \Vert^{2})$. 
The following assertions hold: 
\begin{enumerate}
 \item [{\rm (i)}]$J_{\partial f+\partial g}(0)$ exists 
if and only if there exists $x\in H$ such that 
$\{T^{k}(x)\}$ is bounded, where $T$ is defined by (\ref{NM2}) with $z:=0$, 
$A:=\partial f$ and $B:=\partial g$;
 \item[{\rm (ii)}] 
If $J_{\partial f+\partial g}(0)$ exists, then 
$\{(1/\beta)x_k\}$ converges strongly to $x^{*}$, 
and the convergence rate estimate 
$\Vert (1/\beta)x_{k+1}-x^*\Vert=O(1/k)$ holds, where 
$x^{*}$ is the unique minimizer of (\ref{hoge7}). 
\end{enumerate}
\end{corollary}

\begin{remark}
In \cite{G-H-Y}, 
Guo, Han and Yuan showed the $o(1/\sqrt{k})$ rate of asymptotic regularity 
for the Douglas-Rachford operator \cite[Theorems 5.1 and 5.2]{G-H-Y}. 
However, it does not seem obvious how to estimate the distance between 
the sequences of iterates and the solutions. 
Furthermore, under the metric subregularity assumption \cite[p. 183]{D-R} 
on (\ref{hoge7}), they established the local linear convergence rate 
\cite[Theorem 6.1]{G-H-Y}. In Corollary \ref{Main4}, 
we show that $\Vert (1/\beta)y_{k}-x^*\Vert$ 
converges to zero at the rate of $O(1/k)$ 
without any additional  restrictions
on $\tilde{f}$ and $\tilde{g}$.

\end{remark}

\subsection{Best approximation problems}
Let $C$ and $D$ be closed convex subsets in $H$ with nonempty intersection 
and let $z\in H$. 
Problem (\ref{hoge3}) contains as a special case the 
best approximation problem:
\begin{equation}
\label{BAP} 
\mbox{\rm minimize}~\frac{1}{2}\Vert x-z\Vert^{2}+i_{C}(x)+i_{D}(x),
\end{equation}
where $i_C$ and $i_D$ are the indicator functions of the sets $C$  and $D$.
It is important to point out that it holds $P_{C\cap D}(z) = J_{N_{C}+N_D}(z)$ 
when $J_{N_C+N_D}(z)$ exists. 
(\ref{BAP}) contains a wide variety of problems such as 
covariance design, constrained least-squares matrix 
and signal recovery problems, and 
the analytic expressions for the metric projections onto the constraints sets 
of these problems were developed; 
see, e.g., \cite{S-I-G,Escalante-Raydan,Combettes,C-P} and the references therein.

Now let us apply (\ref{New3}) to problem (\ref{BAP}). 
By letting $f:=i_{C}$ and $g:=i_{D}$, (\ref{New3}) 
is reduced to 
$z_{0}\in H$, 
$x_{0}=\beta  P_{C}\left(
\frac{1}{\beta+r_0(1-\beta)}z_{0}+z\right)-\beta z_{0}$, 
$y_{0}=(1/r_0)(z_0-x_0)$
and 
\begin{equation}
\label{New5}
\left \{
\begin{array}{l}
x_{k}=\beta P_{C}\left(\frac{1}{\beta +r_{k-1}(1-\beta)}
(z_{k-1}+r_{k-1}y_{k-1})+z\right)-\beta z,\\
y_{k}=(1/r_{k-1})(z_{k-1}+r_{k-1}y_{k-1}-x_{k}),\\
z_{k}=
\beta P_{D}\left(\frac{1}{\beta +r_{k}(1-\beta)}
(x_{k}-r_{k}y_{k})+z\right)-\beta z.
\end{array}
\right.
\end{equation}
The evaluation of $P_{C\cap D}(z)$ is in general difficult,  
but each step of our method requires
only the projections $P_C$ and $P_D$ onto the sets $C$ and $D$ respectively. 
We get the following result. 

\begin{corollary}
\label{Main5} 
Let $z\in H$ and let 
$C$ and $D$ be closed convex 
subsets in $H$ with nonempty intersection. 
Assume that $\beta\in (0,1)$ and $\{r_k\}\subset (0,2(1-\beta)/\beta)$ 
such that (\ref{C1}) holds. 
Let $\{x_k\}$, $\{y_k\}$  and $\{z_k\}$ be the sequences generated
by (\ref{New4}).
The following assertions hold:
\begin{enumerate}
\item[{\rm (i)}] $J_{N_C+N_D}(z)$ exists 
if and only if there exists $x\in H$ such that 
$\{T^{k}(x)\}$ is bounded, where $T$ is defined by 
(\ref{NM2}) with $A:=N_{C}$ and $B:=N_{D}$;
\item[{\rm (ii)}]If $J_{N_{C}+N_{D}}(z)$ exists, then 
$\{(1/\beta)x_k+z\}$ converges strongly  to  
$P_{C\cap D}(z)$, and the convergence rate estimate 
$\Vert (1/\beta)x_{k+1}+z-P_{C\cap D}(z)\Vert=O(1/{k})$ holds. 
\end{enumerate}

\end{corollary}

\begin{remark}~
\begin{itemize}
\item[{\rm (i)}] Burachik and Jeyakumar \cite{B-J2} showed that
the normal cone intersection formula $N_{C\cap D}(x)=N_{C}(x)+N_{D}(x)~(\forall 
x\in C\cap D)$ holds whenever $\mbox{\rm epi}\sigma_C+\mbox{\rm epi} \sigma_D$ 
is weakly closed \cite[Theorem 3.1]{B-J2}. 
Furthermore, it was shown that $0\in \mbox{sri}(C-D)$ 
implies $\mbox{\rm epi} \sigma_C+\mbox{\rm epi}\sigma_D$ 
is weakly closed \cite[Proposition 3.1]{B-J2}. 
Arag\'on Artacho and Campoy \cite[Proposition 4.1]{Aragon-Campoy1} showed that 
the normal cone intersection formula is equivalent to 
the following condition: 
\[
 q-P_{C\cap D}(q)\in (N_{C}+N_{D})(P_{C\cap D}(q))~(\forall q\in H). 
\]
Note that under the normal cone intersection formula, 
the assumption of the existence of $J_{N_C+N_D}(z)$ 
in Corollary \ref{Main5} can be removed. 
\item[{\rm (ii)}]  
The convergence results of the averaged
alternating modified reflections method for solving problem (\ref{BAP}) 
were obtained in \cite[Theorem 4.1]{Aragon-Campoy1} 
(see, also \cite[Corollary 3.1]{Aragon-Campoy2}). 
This method is based on the Douglas-Rachford splitting method and 
generates the sequence of iterates which converges 
strongly to $P_{C\cap D}(z)$. However, it seems that the estimate of 
convergence rate for the method has not been considered. 
On the other hand, it is shown that 
the sequence generated by 
the Douglas-Rachford splitting method 
convergences to the solution to (\ref{BAP}) 
with a linear rate when 
$C$ and $D$ are closed subspaces, and $C+D$ is closed \cite{B-B-N-P-W}. 
However, it is worth mentioning that 
the Douglas-Rachford splitting method can be slow 
without such requirements \cite[Section 6]{B-B-N-P-W} 
(see also \cite[Subsection 3.4]{Davis-Yin2} for related results).
(\ref{New4}) can provide 
$O(1/k)$ convergence rate estimate for the distance between 
the sequence of iterates and $P_{C\cap D}(z)$.

\end{itemize}
\end{remark}

\section{Conclusion}
In this paper, we proposed a splitting method for finding the 
resolvent of the sum of two maximal monotone operators. 
Our method is based on 
the accelerated variant of the three operator splitting method
developed in \cite{Davis-Yin2}. 
The method was proved to be strongly
convergent to the solution 
and the $O(1/k)$ convergence rate estimate was also established. 
Finally, we gave some concrete examples and showed how the 
method can be applied to such examples. 

The behavior of the averaged alternating modified reflections algorithm 
can be estimated from the computational experience reported 
 for the best approximation problem of two subspaces 
\cite{Aragon-Campoy1,Aragon-Campoy3} and the continuous-time 
optimal control problem \cite{B-B-K}. 
Numerical results show a very good performance from the 
algorithm, compared to the other existing methods. 
In particular, the numerical results in 
\cite{Aragon-Campoy3} show that the algorithm exhibits 
a linear rate of convergence. 
Hence, it is natural to ask if the linear convergence holds 
when (\ref{New5}) is applied to two subspaces,
or to give a counter example, if it does not. 
We leave this as one of our future research topics.

\section*{Acknowledgments}
The author is grateful to Professors W. Takahashi of Tokyo Institute of Technology,
D. Kuroiwa of Shimane University and Li Xu of Akita Prefectural
University for their helpful support.

\end{document}